\documentclass[11pt]{amsart}

\usepackage[utf8]{inputenc}
\usepackage[english]{babel}
 
\usepackage{anysize}
\usepackage{amsfonts}
\usepackage{amssymb}
\usepackage{amsmath}
\usepackage{amsthm}
\usepackage[alphabetic,backrefs,lite]{amsrefs} 
\usepackage[all]{xy}
\usepackage{enumerate}
\usepackage{multirow}
\usepackage{geometry}
\usepackage{color}
\usepackage{colonequals}
\usepackage{mathrsfs}

\definecolor{darkgreen}{rgb}{0,0.5,0}

\usepackage[
    draft=false,
    colorlinks, citecolor=darkgreen,
    backref,       
    pdfauthor={Sonia Brivio and Filippo F. Favale},
    pdftitle={Higher-rank Brill-Noether loci on nodal reducible curves},
    linktocpage]
{hyperref}

\newtheorem{theorem}{Theorem}[section]
\newtheorem*{theorem*}{Theorem}
\newtheorem{proposition}[theorem]{Proposition}

\newtheorem{lemma}[theorem]{Lemma}

\newtheorem{remark}[theorem]{Remark}
\newtheorem{example}[theorem]{Example}
\newtheorem{conjecture}[theorem]{Conjecture}

\DeclareMathOperator{\SL}{SL}

\DeclareMathOperator{\coker}{coker}
\DeclareMathOperator{\Supp}{Supp}
\DeclareMathOperator{\Sing}{Sing}

\DeclareMathOperator{\Ext}{Ext}
\DeclareMathOperator{\Hom}{Hom}
\DeclareMathOperator{\Pic}{Pic}

\DeclareMathOperator{\rank}{rank}

\DeclareMathOperator{\rk}{rk}

\newcommand{\N}{\mathbb{N}}

\newcommand{\Q}{\mathbb{Q}}

\newcommand{\R}{\mathbb{R}}

\newcommand{\OO}{\mathcal{O}}

\newcommand{\U}{\mathcal{U}}

\newcommand{\ev}{ev}

\DeclareMathOperator{\w}{\underline{w}}
\DeclareMathOperator{\wa}{(\w,\alpha)}

\DeclareMathOperator{\wdeg}{deg_{\underline{w}}}
\DeclareMathOperator{\wrank}{rk_{\underline{w}}}
\DeclareMathOperator{\Gr}{Gr}

\newcommand{\Gwa}{\mathcal{G}_{(C,\w),\alpha}}
\newcommand{\GwL}{\mathcal{G}_{(C,\w),L}}

\newcommand{\Bw}{\mathcal{B}_{(C,\w)}}
\newcommand{\Uw}{\mathcal{U}_{(C,\w)}}

\newcommand{\Btw}{\tilde{\mathcal{B}}_{(C,\w)}}

\DeclareMathOperator{\urank}{\underline{rk}}
\newcommand{\ur}{\underline{r}}

\newcommand{\nc}{\newcommand}

\newcommand{\un}[1]{\underline{#1}}

\nc{\cH}{{\mathcal H}}
\nc{\cA}{{\mathcal A}}
\nc{\cG}{{\mathcal G}}
\nc{\cC}{{\mathcal C}}
\nc{\cD}{{\mathcal D}}
\nc{\cO}{{\mathcal O}}
\nc{\cI}{{\mathcal I}}
\nc{\cB}{{\mathcal B}}
\nc{\cY}{{\mathcal Y}}
\nc{\cK}{{\mathcal K}} 
\nc{\cX}{{\mathcal X}}
\nc{\cS}{{\mathcal S}}
\nc{\cE}{{\mathcal E}}
\nc{\cF}{{\mathcal F}}
\nc{\cZ}{{\mathcal Z}}
\nc{\cQ}{{\mathcal Q}}
\nc{\cN}{{\mathcal N}}
\nc{\cP}{{\mathcal P}}
\nc{\cL}{{\mathcal L}}
\nc{\cM}{{\mathcal M}}
\nc{\cT}{{\mathcal T}}
\nc{\cW}{{\mathcal W}}
\nc{\cU}{{\mathcal U}}
\nc{\cJ}{{\mathcal J}}
\nc{\cV}{{\mathcal V}}
\nc{\cR}{{\mathcal R}}
\nc{\bH}{{\mathbb H}}
\nc{\bA}{{\mathbb A}}
\nc{\bG}{{\mathbb G}}
\nc{\bC}{{\mathbb C}}
\nc{\bO}{{\mathbb O}}
\nc{\bI}{{\mathbb I}}
\nc{\bB}{{\mathbb B}}
\nc{\bY}{{\mathbb Y}}
\nc{\bK}{{\mathbb K}} 
\nc{\bX}{{\mathbb X}}
\nc{\bS}{{\mathbb S}}
\nc{\bE}{{\mathbb E}}
\nc{\bF}{{\mathbb F}}
\nc{\bZ}{{\mathbb Z}}
\nc{\bQ}{{\mathbb Q}}
\nc{\bN}{{\mathbb N}}
\nc{\bP}{{\mathbb P}}
\nc{\bL}{{\mathbb L}}
\nc{\bM}{{\mathbb M}}
\nc{\bT}{{\mathbb T}}
\nc{\bW}{{\mathbb W}}
\nc{\bU}{{\mathbb U}}
\nc{\bD}{{\mathbb D}}
\nc{\bJ}{{\mathbb J}}
\nc{\bV}{{\mathbb V}}
\nc{\bR}{{\mathbb R}}

\newcommand{\cHom}{\mathcal{H}om}
\newcommand{\cExt}{\mathcal{E}xt}

\numberwithin{equation}{section}

\begin{document}

\title[Higher-rank Brill-Noether loci on nodal reducible curves]
{Higher-rank Brill-Noether loci on nodal reducible curves}

\author{Sonia Brivio}
\address{Dipartimento di Matematica e Applicazioni,
	Universit\`a degli Studi di Milano-Bicocca,
	Via Roberto Cozzi, 55,
	I-20125 Milano, Italy}
\email{sonia.brivio@unimib.it}

\author{Filippo F. Favale}
\address{Dipartimento di Matematica,
	Universit\`a degli Studi di Pavia,
	via Ferrata, 5
	I-27100 Pavia, Italy}
\email{filippo.favale@unipv.it}

\date{\today}
\thanks{
\\
Both authors are partially supported by INdAM - GNSAGA. The second author is partially supported by PRIN 2017 \emph{``Moduli spaces and Lie theory''} and by (MIUR): Dipartimenti di Eccellenza Program (2018-2022) - Dept. of Math. Univ. of Pavia. \\
}

\begin{abstract}
In this paper, we deal with Brill-Noether theory for higher-rank sheaves on a polarized nodal reducible curve $(C,\w)$ following the ideas of \cite{BGN97}. We study the Brill-Noether loci of $\w$-stable depth one sheaves on $C$ having rank $r$ on all irreducible components and having small slope. In analogy with what happens in the smooth case, we prove that these loci are closely related to  BGN extensions. Moreover, we produce irreducible components of the expected dimension for these Brill-Noether loci. 
\end{abstract}

\maketitle

\section*{Introduction}
Classical Brill-Noether theory was born in the last century 
in order to describe
the subchemes $W^{k-1}_d$ of $\Pic^d(C)$ parametrizing degree $d$-line bundles
on a smooth curve $C$ having  at least $k$ linearly independent global sections. Geometric properties of these loci (such as non-emptyness, irreducibily, connectedess, dimension and singularities) have been completely studied at least for a general curve. For a full treatment of the topic see \cite{ACGH}. 
\medskip

The notion of Brill-Noether locus has been extended in the years to vector bundles of higher-rank (see \cite{MerP} for an historical overview). These loci are  closed subschemes of the moduli space $\cU_{C}^s(r,d)$ parametrizing stable vector bundles of rank $r$ and degree $d$ on a smooth curve $C$.
More precisely, the Brill-Noether locus 
$$\cB_C(r,d,k)=\{[E]\in \cU_C^s(r,d)\,|\, h^0(E)\geq k\},$$
parametrizes isomorphism classes of vector bundles of rank $r$ and degree $d$ with at least $k$ independent global sections. A similar definition has been introduced for equivalence classes of semitable vector bundles.
\medskip

The higher-rank case is far from beeing completely understood. We recall some of the first results about the geometry of these loci which are related to the content of our paper.  The case $k=1$ has been studied  by  \cite{Sun} and \cite{Lau}. More general results are due to Teixidor i Bigas (see \cite{TeiBN1,TeiBN2}) while Brambila-Paz, Grzegorczyk and Newstead have studied Brill-Noether loci for vector bundles with small slope (see \cite{BGN97}). 
In particular, we are interested in the  following seminal result:
\begin{theorem*}[Theorems $A +B$ of \cite{BGN97}]
Let $C$ be a smooth curve of genus $g$. Let $r\geq 2$ and $0 \leq d\leq r$. Then, the Brill-Noether locus $\cB_C(r,d,k)$ is non-empty if and only $$d > 0, \quad kg\leq r(g-1)+d\quad \mbox{ and }\quad (d,k)\neq (r,r).$$ 
Under these assumptions, it is irreducible of dimension equal to the Brill-Noether number $\beta_C(r,d,k)$ and $\Sing(\cB_C(r,d,k))= \cB_C(r,d,k+1)$.
\end{theorem*}
We recall that in the above result,  Brill-Noether loci are described as spaces of particular  extensions of a semistable sheaf by a trivial one, which have been introduced in \cite{BGN97} and in the sequel  have been called BGN estensions (see \cite{BG02}).
\medskip

Brill-Noether theory for higher-rank extends naturally to the case of nodal irreducible curves by considering  stable torsion free sheaves and their moduli spaces (see \cite{Bho07}). In particular, in \cite{Bho07},
Theorems $A+B$ of \cite{BGN97} have been extended almost completely for any nodal irreducible curve. 
\medskip

In this paper we deal with Brill-Noether theory for higher-rank on nodal reducible curves following the ideas of \cite{BGN97}.

As in the irreducible nodal case, one can not consider only locally free sheaves in order to construct compact moduli spaces: one need to take into account also depth one sheaves. Moreover, one have to choose a polarization $\w$ on the curve in order to have moduli spaces for these sheaves.
For details, one can refer to subsection \ref{Subsec:depthone}.
If $\w$ is a polarization on a nodal reducible curve $C$, we denote by $\Uw^s(\ur,d)$ the moduli space of isomorphism classes of $\w$-stable depth one sheaves on $C$ with  multirank $\ur$ and $\w$-degree $d$. Then, for any integer $k \geq 1$, the Brill-Noether loci can be naturally defined as the following subsets:
$$\Bw(\ur,d,k) = \{ [F] \in \Uw^s(\ur,d) \,\vert\, h^0(F) \geq k \}.$$
A similar definition can be given by considering the moduli space $\Uw(\ur,d)$, parametrizing equivalence classes of $\w$-semistable depth one sheaves on $C$ (see Section \ref{SEC:2}).

The description of these subsets given by Mercat in the smooth case works, with  necessary technical adjustment, even for a reducible nodal curve, so  we obtain  a closed subscheme structure for these loci (see Proposition \ref{PROP:scheme}). 
Unfortunately, even for an irreducible nodal curve the local study cannot be carried out as in the smooth case unless we consider a locally free sheaf (see \cite{Bho07} and \cite{Bho}).
For this reason,  we restrict our attention to depth one sheaves having rank $r$ on all irreducible component of $C$. By doing this, we have a moduli space 
$\Uw^s(r\cdot \underline{1},d)$ whose general element is a locally free sheaf of
rank $r$ and  degree $d$ (see \cite{Tei91} and \cite{Tei95}). 
Here, as in the smooth case, we can define the Brill-Noether number
$\beta_C(r,d,k)$ and the local study of smoothness  at  a locally free sheaf in $\Bw(r\cdot \underline{1},d,k)$  can be done as in the smooth case (see Proposition  \ref{PROP:Petri}).
\medskip

The purpose of this paper is to understand whether the results of Theorems $A+B$ of \cite{BGN97} can be extended to nodal reducible curves. 
We recall that Theorems $A+B$ imply that the elements of the Brill-Noether loci  for small slope (i.e.
$0 \leq d \leq r$) are {\it all} given by BGN extensions. 
The authors of this paper generalized the notion of BGN extension to the case of depth one sheaves on a nodal reducible curve in \cite{BFBGN} and they described the space parametrizing BGN etensions (as a moduli space of coherent systems).
The behaviour of these spaces is extremely wild, unless one chooses a polarization which is {\it good} (see \cite{BFPol} or Section \ref{SEC:2} for details).
Unfortunately, also by working with good polarizations, not all elements of the Brill-Noether loci are given by BGN extensions, as Example \ref{EX:noBGN} shows. Nevertheless, we prove that this holds when we consider locally free sheaves by giving the following partial generalization of Theorems $A+B$:

\begin{theorem*}[Theorem \ref{THM:genTHMb}]
Let $(C,\w)$ be a polarized nodal curve with $\w$ good. Let $r,k,d\in \bN$ such that $r \geq 2$, $k \geq 1$ and $d \geq 0$. Let $E$ be a locally free sheaf in $\Bw( r\cdot \underline{1},d,k)$ which satisfies at least one of the following two conditions:
\begin{enumerate}[(a)]
    \item $ 0 \leq d\leq r$;
    \item for any  irreducible component $C_i$  of $C$, the restriction $E|_{C_i}$ is stable and $ 0 \leq \deg(E|_{C_i})\leq r$.
\end{enumerate}
Then 
$$d>0,\qquad k p_a(C) \leq r(p_a(C)-1)+d$$ 
and $E$ is obtained as a BGN extension of a locally free sheaf of rank $r-k$.
\end{theorem*}

Let $E$ be a locally free sheaf on $(C,\w)$ of rank $r$ and  degree $d$. We say that $E$ has {\it small slope} if either $0\leq d\leq r$ or if we have $0 \leq \deg(E|_{C_i})\leq r$ for any irreducible component $C_i$ of $C$. Then, the above theorem can be seen in the framework of Brill-Noether theory for locally free sheaves of small slope.
\medskip

In the second part of this paper, under numerical  assumptions on $d$, $r$, $k$,  we give  a method  to construct irreducible components of Brill-Noether loci for sheaves of small slope, using BGN extensions.
In order to do so, we study $\w$-stable BGN extensions defined by 
irreducible components of  moduli spaces of $\w$-stable sheaves of small slope. The details are rather technical: we refer to Proposition  
\ref{PROP:Estab} and Theorem \ref{THM:Main}.
In Section \ref{SEC:SMALLSLOPE} we give sufficient conditions for the existence of components of the moduli spaces of depth one sheaves with small slope. These conditions are stated in Proposition \ref{PROP:TEC}. Then, using the above technical results, we prove our second main theorem: 

\begin{theorem*}(Theorem \ref{THM:MAIN-B})
Let $C$ be either a chain-like or comb-like curve with $\gamma$ irreducible components of genus $g_i \geq 2$.
Let  $d,s,k \in \mathbb N$  such that $$k \leq  1 + s(g_i -1)\quad \mbox{ for all }\quad  i=1,\dots, \gamma,\qquad s\geq 2(\gamma-1)\quad \mbox{ and }\quad \gamma\leq d\leq s.$$
Then, the Brill-Noether locus $\Bw((s+k)\cdot \underline{1},d,k)$ is non-empty whenever $\w$ lies in a suitable open neighborhood of the canonical polarization. Moreover, it has an irreducible component of dimension $\beta_C(s+k,d,k)$.
\end{theorem*}

We stress that Theorem \ref{THM:MAIN-B} is stated in Section \ref{SEC:SMALLSLOPE} for a wider class of curves (more precisely, for curves satisfying one of the conditions in Proposition \ref{PROP:TEC}). We reported it here in this form for brevity.

The above results give a partial generalization of Theorems $A+B$. Nevertheless, we conjecture that the Brill-Noether locus $\Bw((s+k)\cdot \underline{1},d,k)$ is non-empty whenever 
$$2(\gamma-1)\leq s,\quad  \gamma\leq d\leq s \qquad \mbox{ and }\qquad k g_i \leq 1+s(g_i-1), \forall i = 1,\dots \gamma,$$
for any nodal curve $C$ of compact type and $\w$ in a suitable neighborhood of the canonical polarization   (see Conjecture \ref{CONJ:STABJ}).


\section{Notations and preliminary results}
\subsection{Nodal curves}
\label{SUBSEC:nodalcurves}
In this paper we will deal with connected nodal reducible curves over the complex field. A comprehensive reference for general theory on nodal curves are \cite{Cap} and \cite[Ch X]{ACG}. If $C$ is as above, we will denote by $\gamma$ the number of its irreducible components and by $\delta$ the number of its nodes. We will assume that each irreducible component $C_i$ is a smooth curve of genus $g_i \geq 2$.  


The {\it arithmetic  genus} of $C$ is   
\begin{equation}
\label{EQ:paC}
p_a(C)=\sum_{i=1}^{\gamma}g_i+\delta-\gamma+1.
\end{equation}
\noindent For any subcurve $B$ of $C$, let $B^c$ be the closure of $C\setminus B$. We set $\Delta_i=C_i\cap C_i^c$, and we denote by  $\delta_i$ its  degree, i.e. the number of nodes of $C$ on $C_i$. 
We recall that there exists on $C$ a dualizing sheaf $\omega_C$  which is an invertible sheaf, moreover for  any $i= 1 \dots \gamma$, we have  $\omega_C|_{C_i}=\omega_{C_i}(\Delta_i)$. Since $C$ is a nodal curve without rational or elliptic components, we have that $C$ is a {\it stable} curve. In particular, $\omega_C$ is an ample line bundle. 
The curve $C$ is said {\it of compact type} if its {\it dual graph} is a tree. In this case  we have $\delta = \gamma -1$ and $p_a(C)= \sum_{i=1}^{\gamma}g_i$. \vspace{2mm}

The following lemma gives a technical result useful for the sequel. It is a small improvement of \cite[Lemma 1]{Tei91}.

\begin{lemma}
\label{LEM:ORDER}
Let $C$ be a nodal curve of compact type with $\gamma$ irreducible components. Fix an irreducible component $D$ of $C$, it is possible to order the components of $C$ and to give a family of subcurves $\{A_j\}_{j=1,\dots, \gamma -1}$ of $C$ such that:
\begin{enumerate}[(a)]
\item{} $C_\gamma$ is the chosen component, i.e. $C_{\gamma}=D$;
\item{} for any $i=1,\dots,\gamma-1$ the curve $C_{i+1} \cup \dots  \cup C_{\gamma}$ is connected;
\item{}  for any $i=1,\dots, \gamma-1$, $C_i \subseteq A_i$, $A_i$ and $A_i^c$ are connected.
\end{enumerate}
In particular, this implies that $A_i \cap A_i^c$ is a node: we denote it by $p_i$.
\end{lemma}

\begin{proof}
We proceed by induction on $\gamma$. If $\gamma=2$ the result is straightforward. We assume by induction hypothesis that the result holds for any curve of compact type with at most $\gamma-1$ irreducible components. Fix a component of $C$ and denote it by $C_\gamma$. Let $m:=C_{\gamma}\cdot C_{\gamma}^c$, i.e. $m$ is the number of nodes on $C_\gamma$. Then $C_{\gamma}^c$ has $m$ connected components which will be denoted by $\Gamma^{(1)},\dots,\Gamma^{(m)}$.  The ordering of these components is arbitrary.
By construction $\Gamma^{(k)}$ is a curve of compact type with less than $\gamma$ components. Since $C$ is of compact type, for any  $k=1,\dots, m$ we have that $\Gamma^{(k)}\cap C_{\gamma}$ is a single point. Then, there exists a unique $B_k$, irreducible component of $\Gamma^{(k)}$, such that $B_k\cap C_{\gamma}$ is not empty. By induction hypothesis we have an order of the components of $\Gamma^{(k)}$ whose "final" component is $B_k$ satisfying $(b)$.
The ordering on each $\Gamma^{(k)}$ induce a natural ordered sequence of all the components of $C$ whose last element is $C_{\gamma}$. 
\vspace{2mm}

We now check that this ordering satisfies $(b)$. For any $i\leq \gamma-1$, $C_i$ is contained in a unique $\Gamma^{(k_i)}$. If $C_i=B_{k_i}$ then $C_{i+1}\cup\cdots \cup C_{\gamma}=\bigcup_{k=k_i+1}^{m} \Gamma^{(k)}\cup C_{\gamma}$, so it is connected. If $C_i\neq B_{k_i}$ then, since $B_{k_i}=C_{l}$ for a unique index $l>i$, we have that $\bigcup_{j=i+1}^{l}C_j=C_{i+1}\cup \cdots \cup C_{l}$ is a connected subcurve of $\Gamma^{k_i}$ by induction hypothesis and meets $C_{\gamma}$ in a point.
Then $$\bigcup_{j=i+1}^{\gamma}C_j = \left(\bigcup_{j=i+1}^{l}C_j\right)\cup \bigcup_{k=k_i+1}^{m} \Gamma^{(k)}\cup C_{\gamma}$$ is connected.
\vspace{2mm}

$(c)$ follows from $(b)$. In fact,  for any  $i=1,\dots, \gamma-1$, let $D_i$ be the connected component of $C_i^{c}$ containing $C_{i+1}\cup \cdots \cup C_{\gamma}$. We define $A_i$ to be $D_i^c$. By construction, it contains $C_i$ and all the possible other connected components of $C_i^c$ different from $D_i$. Hence, $A_i$ is a connected curve and $C_j\subseteq A_i$ implies $j\leq i$. In particular, if $C_i\subset \Gamma^{(k)}$ then $C_{\gamma}\cup\bigcup_{j\neq k} \Gamma^{j}\subseteq A_i^c$ so $A_i\subseteq \Gamma^{(k)}$.
\end{proof}

\begin{example}[Chain-like curves]
\label{EX:chain}
A {\it "chain-like"} curve is a curve of compact type with $\gamma \geq 2$ smooth irreducible components which can be ordered as $\{C_1, \dots, C_{\gamma}\}$ with $C_i\cap C_{i+1}=\{p_i\}$ and $C_i\cap C_j=\emptyset$ whenever $|i-j|>1$. This ordering satisfies  conditions $(b)$ and $(c)$ of Lemma \ref{LEM:ORDER} and it is obtained by choosing as $C_{\gamma}$ one of the two components of $C$  having a single node.  It is a "natural" ordering for chain-like curves as it 
gives $A_j=\bigcup_{i=1}^{j}C_i$ for $j=1,\dots, \gamma-1$. On the other  hand, for any $i=1,\dots, \gamma-1$, one can also chose and alternative ordering $\{\tilde{C}_1,\dots, \tilde{C}_{\gamma}\}$ with $\tilde{C}_{\gamma}=C_i$. If $i=1$ we are simply reversing the ordering of the curves. If $i>1$, we have that $C_i^c$ has two irreducible components
$\Gamma^{(1)}=\bigcup_{j=1}^{i-1}C_i$ and $\Gamma^{(2)}=\bigcup_{j=i+1}^{\gamma}C_i$. Using the notation introduced in the proof of the Lemma, we have $B_1=C_{i-1}$ and $B_2=C_{i+1}$ so we have $$\{\tilde{C}_1,\dots, \tilde{C}_{\gamma}\}=\{C_1,\dots, C_{i-1},C_{\gamma},C_{\gamma-1},\dots,C_{i+1},C_{i}\}.$$
\end{example}

\begin{example}[Comb-like curves]
\label{EX:comb}
A {\it "comb-like curve"} with $\gamma\geq 2$ smooth irreducible components is a curve of compact type where all the nodes lies on a single component (the "grip" of the curve), i.e. with a 
component with $\gamma-1$ nodes. Its components can be ordered as $\{C_1, \dots, C_{\gamma}\}$ with $C_\gamma \cdot C_i=\{p_i\}$ for $i=1,\dots, \gamma-1$ and $C_i\cap C_j=\emptyset$ whenever $i\neq j$ and $i,j\leq \gamma-1$. This ordering satisfies  conditions (b) and (c) of  Lemma \ref{LEM:ORDER} and we have $A_i=C_{i}$ for all $i=1,\dots, \gamma-1$. Any permutation of the indices $\{1,\dots, \gamma-1\}$ gives an analogous result. Starting from the above ordering, one can also chose the ordering $\{C_2,C_3,\dots,C_{\gamma-1},C_{\gamma},C_1\}$ which yields $A_{i}=C_{i+1}$ for $i\leq \gamma-2$ and $A_{\gamma-1}=\bigcup_{j=2}^{\gamma}C_i$.
\end{example}

Finally, we recall some  general technical  results.
Let $p$ be a node and denote by $C_{i_1}$ and $C_{i_2}$ the two components such that $p\in C_{i_1}\cap C_{i_2}$. Following the notations of \cite{Ses}, chap. 8,  we set:
$$\OO_{x_{i_k}}= \OO_{C_{i_k},p}, \quad m_{x_{i_k}} = m_{C_{i_k},p}, \quad \OO_p= \OO_{C,p} \quad m_{p} = m_{C,p}.$$
Then:
$$\OO_p = \{ (f,g) \in \OO_{x_{i_1}} \oplus \OO_{x_{i_2}} \,\vert\,  f(p) = g(p) \}, \quad m_p = m_{x_{i_1}}
\oplus m_{x_{i_2}}.$$
The  
 isomorphisms $\OO_{x_{i_k}} \simeq m_{x_{i_k}}$ 
obtained by sending $f \mapsto ft_{i_k}$, where $t_{i_k}$ is a local coordinate on $C_{i_k}$ at $p$, induce an isomorphism 
$\OO_{x_{i_1}} \oplus \OO_{x_{i_2}} \simeq m_p$. 
We have the  following exact sequences of $\OO_p$-moduli:
\begin{equation}
\label{eq1}
0 \to \OO_p \to \OO_{x_{i_1}} \oplus \OO_{x_{i_2}} \to {\mathbb C} \to 0\qquad \mbox{ and }\qquad 0 \to m_p \to \OO_p \to {\mathbb C} \to 0.
\end{equation}

The above exact sequences and standard facts on modules yield the following lemma giving $Ext$-groups for $\cO_p$-modules of depth one where $p$ is as above. Some of these groups has been computed in \cite[Lemma 2.1]{BFBGN}.

\begin{lemma} 
\label{homloc}
Let $N$ be a $\cO_p$-module. Then, the following facts hold:
$$
\Ext^1(\cO_p,N)=\Ext^1(\cO_{x_i},\cO_p)=\Ext^1(\cO_{x_i},\cO_{x_i})=0\quad \mbox{ and }\quad \Ext^1(\cO_{x_i},\cO_{x_j})=\bC  \quad\mbox{ for } i\neq j.
$$

If $M \simeq \cO_p^s \oplus \cO_{x_1}^{a_1} \oplus \cO_{x_2}^{a_2}$ and $N \simeq \cO_p^{s'} \oplus \OO_{x_1}^{b_1} \oplus \cO_{x_2}^{b_2}$, we have $\Ext^1_{\cO_p}(M,N) \simeq \bC^{a_1b_2}\oplus \bC^{a_2b_1}$. In particular, if either $M$ or $N$ is free we have $\Ext^1(M,N)=0$.
\end{lemma}

\subsection{Depth one sheaves on nodal curves and related moduli spaces}
\label{Subsec:depthone}
We recall the notion of depth one sheaves on nodal curves.  References for the contents of this subsection are \cite{Ses} and \cite{KN}.
\medskip

A coherent sheaf $E$ on a reduced curve  is said to be of {\it depth one} if  for any $x \in \Supp(E)$ the stalk $E_x$ is an $\OO_x$-module of depth one. 
Let $C$ be a nodal curve with smooth irreducible components $C_1,\dots,C_{\gamma}$. Using the notations introduced above, a coherent sheaf $E$ on $C$ is of depth one if  $E$ is locally free away from the nodes and the stalk of $E$ at a node $p \in C_{i_1}\cap C_{i_2}$ is isomorphic to $\OO_p^s \oplus \OO_{x_{i_1}}^{a_1} \oplus \OO_{x_{i_2}}^{a_2}$. In particular, vector bundles  are depth one sheaves and any subsheaf of a  depth one sheaf is of depth one too.
\hfill\par

Let $E$ be a depth one sheaf on  $C$. Its dual sheaf $E^* =  \cHom_{\OO_C}(E,\OO_C)$  is  of depth one too  and $E$ is reflexive, i.e. $\cHom_{\OO_C}(E^*,\OO_C) \simeq E$.   
In particular, we recall that Serre duality   yields  an isomorphism  $H^q(E)^* \simeq H^{1-q}(E^* \otimes {\omega}_C)$  for any $q \geq 0$.
\medskip



The following Lemma generalizes formula in \cite[Lemma 2.5]{Bho06} to the case of nodal reducible curves.

\begin{lemma}
\label{LEM:EXT1}
Let $E$  and $F$ be  depth one sheaves on $C$.
Assume that at the node $p_j\in C_{j,1}\cap C_{j,2}$ we have 
$$E_{p_j} \simeq \OO_p^{s_j} \oplus \OO_{x_{j,1}}^{a_{j,1}} \oplus \OO_{x_{j,2}}^{a_{j,2}}\qquad \mbox{ and }\qquad F_{p_j} \simeq \OO_p^{t_j} \oplus \OO_{x_{j,1}}^{b_{j,1}} \oplus \OO_{x_{j,2}}^{b_{j,2}},$$
then $$\dim \Ext^1(E,F) = h^1(\cHom(E,F)) + \sum_{j=1}^{\delta}(a_{1,j}b_{2,j}+a_{2,j}b_{1,j}).
$$
\end{lemma}
\begin{proof}
For all $q\geq 1$ we have that $\cExt^q(E,F)$ is a torsion sheaf, whose support in contained in the set of nodes, while $\cExt^0(E,F) = \cHom(E,F)$. In particular, the cohomology group $H^p(\cExt^q(E,F))$ vanish if either $p=1$ and $q\geq 1$ or $p\geq 2$ for all $q\geq 0$. Then, the local-to-global spectral sequence for $\Ext$ groups (see \cite{God73}) yields an exact sequence
$$0\to H^1(\cHom(E,F))\to \Ext^1(E,F)\to H^0(\cExt^1(E,F))\to H^2(\cHom(E,F))=0$$
so $\dim \Ext^1(E,F) = h^1(\cHom(E,F)) + h^0(\cExt^1(E,F))$. Since $$h^0(\cExt^1(E,F))=\bigoplus_{j=1}^{\delta}\dim(\Ext^1(E_{p_j},F_{p_j})),$$ 
one can concludes using Lemma \ref{homloc}. 
\end{proof}

In order to introduce moduli spaces for depth one sheaves on a reducible curve it is necessary to introduce the notion of polarization. A {\it polarization} on a nodal reducible curve $C$ is a vector  $\w= (w_1,\dots,w_{\gamma}) \in {\mathbb Q}^{\gamma}$ such that
\begin{equation}
0 < w_i < 1 \quad \sum_{i=1}^{\gamma}w_i = 1.
\end{equation}
We will say that the pair $(C,\w)$ is a {\it polarized nodal curve}. Any ample line bundle $L$ on $C$ induces a polarization $\w_L$ whose weight on the component $C_i$ is $\deg(L|_{C_i})/\deg(L)$. 
\vspace{2mm}

Let $(C,\w)$ be a polarized nodal curve. 
For any depth  one sheaf $E$ on $C$ we denote by $E_i$ its restriction to $C_i$ modulo torsion and  by $\urank(E)=\ur=(r_1,r_2,\dots,r_{\gamma})$ its {\it multirank}, where $r_i=\rank(E_i)$.  We define the {\it $\w$-rank} and the {\it $\w$-degree} of $E$ as: $$\wrank(E)=\sum_{i=1}^{r}r_iw_i \quad  \mbox{ and }\quad \wdeg(E)=\chi(E)-\wrank(E)\chi(\OO_C).$$ 
The {\it $\w$-slope} of $E$ is defined as
$\mu_{\w}(E) =  \wdeg(E)/\wrank(E)$. $E$ is said {\it${\w}$-semistable} ({\it $\w$-stable} respectively) if  for any proper subsheaf $F$ of $E$ we have $\mu_{\w}(F) \leq \mu_{\w}(E)$ ($\mu_{\w}(F) < \mu_{\w}(E)$ respectively).
We denote by $\Uw^s(\ur,d)$ the moduli space parametrizing isomorphism classes of $\w$-stable depth one sheaves on $C$ with  multirank $\ur$ and $\w$-degree $d$ and by $\Uw(\ur,d)$ its compactification, which is obtained by considering   $S$-equivalence classes of $\w$-semistable depth one sheaves. 
\vspace{2mm}

For any depth one sheaf $E$ we define 
\begin{equation}
\label{EQ:DELTA}
\Delta_{\w}(E)=\wdeg(E)-\sum_{i=1}^{\gamma}\deg(E_i).    
\end{equation}
We say that $\w$ is a {\it good polarization}  on $C$ if $\Delta_{\w}(E)\geq 0$ for all depth-one sheaves and equality holds if and only if $E$ is locally free. Good polarizations were introduced in \cite{BFPol},  where the authors proved that good polarizations always exist on any stable nodal curve $C$ with $p_a(C) \geq 2$. Moreover, if $\w$ is good, then $\cO_C$ is $\w$-stable and the converse holds when $C$ is a nodal curve of compact type (see \cite[Theorem 3.10]{BFPol}). It is also conjectured that this should hold for any nodal curve.  
\vspace{2mm}

Finally, we recall the notion of coherent system on a polarized nodal curve $(C,\w)$ (see \cite{BFCoh} for details). We refer to \cite{BGMN} for treatment of the smooth case. A {\it coherent system} is given by a pair $(E,V)$, where $E$ is a depth one sheaf on $C$  and $V$ is a subspace of $H^0(E)$. If $\wrank(E) = r$, $\wdeg(E)= d$ and  $\dim V=k$ it is said of type $(r,d,k)$ (and of multitype $(\ur,d,k)$ if $\urank(E) = \ur$).

For any  $\alpha \in \R$, the {\it $\wa$-slope}  of $(E,V)$ is defined as 
$$\mu_{\w,\alpha} (E,V) = \mu_{\w}(E) + \alpha \dim(V)/\wrank(E).$$   
$(E,V)$ is  said  to be $\wa$-stable if for any proper coherent subsystem $(F,U)$  of $(E,V)$ we have
$\mu_{\w,\alpha} (F,U) < \mu_{\w,\alpha} (E,V)$. 
We denote by 
${\mathcal G}_{(C,\w),\alpha}(r,d,k)$  the  moduli space parametrizing $\wa$-stable  coherent systems of type $(r,d,k)$. If we fix  $\ur=(r_1,\dots,r_{\gamma})$, we obtain the moduli space $\Gwa(\ur,d,k)$,  which is a component of the previous one. For more details one can see \cite{BFCoh}. 
\vspace{2mm}

In this paper we will assume  $k < r$.
In \cite{BFBGN} the authors proved that for any $\w$ there exists $M_{\w}>0$ such that  $\Gwa(\ur,d,k)$ is  empty  whenever $\alpha\not\in(0,M_{\w})$. Moreover, there are a finite number of values $0<\alpha_1<\cdots<\alpha_L<M_{\w}$, called critical values, such that, the property of $\wa$-stability is independent on the choice of $\alpha \in(\alpha_i,\alpha_{i+1})$. Hence, for fixed $\w$, there are up to finitely many different and not empty moduli spaces $\Gwa(\ur,d,k)$.
We denote by $\GwL(\ur,d,k)$ the ``terminal'' moduli space, the one obtained by considering $\alpha\in (\alpha_L,M_{\w})$. If $\w$ is a good  polarization, then  $M_{\w}=d/(r-k)$ and hence  $d>0$, see \cite{BFBGN}. In the same paper, these spaces have been described using BGN extensions (in analogy of what happens for the smooth case in \cite{BGN97}).
We recall that a {\it BGN estension} of type $(r,d,k)$ on $(C,\w)$ is an extension
\begin{equation}
\label{EQ:BGNext}
\underline{e}:\qquad 0 \to V \otimes \OO_C \to E \to F \to 0,
\end{equation}
where $V$ is a vector space of dimension $k$, $F$ is a depth one sheaf on $C$ with $\wdeg(F)= d$ and $\wrank(F)= r-k$ and ${\underline e}=(e_1,\dots, e_k)\in \Ext^1(F,V\otimes \cO_C)\simeq \Ext^1(F,\cO_C)^{\oplus k}$ is such that $\{e_1,\dots,e_k\}$ are linearly independent.

\section{Brill-Noether loci on nodal reducible curves}
\label{SEC:2}

Let $(C,\w)$ be a polarized nodal curve. Brill-Noether loci can be defined in analogy with the smooth case as follows. For any $d\in \Q, \ur= (r_1,\dots,r_{\gamma})\in \bN^\gamma$ and $k \geq 1$ we define set-theoretically the {\bf Brill-Noether loci} as:
$$\Bw(\ur,d,k) = \{ [F] \in \Uw^s(\ur,d) \,\vert\, h^0(F) \geq k \},$$
$$\Btw(\ur,d,k) = \{ [F] \in \Uw(\ur,d) \,\vert\, h^0(gr(F)) \geq k \}.$$
When $C$ is nodal but irreducible, these spaces have been introduced and studied in \cite{Bho07}.

\begin{proposition}
\label{PROP:scheme}
$\Bw(\ur,d,k)$ is a closed subscheme of the moduli space $\Uw^s(\ur,d)$. If it is non-empty, let $Z$ be any irreducible component of $\Bw(\ur,d,k)$ and denote by $X_Z$ the irreducible component of $\Uw^s(\ur,d)$ containing $Z$. Then $Z$ has codimension at most 
$k(k-d +r(p_a(C)-1))$ in $X_Z$, where $r=\sum_{i=1}^{\gamma}w_ir_i$.
\end{proposition}

\begin{proof}
In order to give a subscheme structure to the above subsets we follow the approach of Mercat in the case of smooth curves (see \cite{MerP,Mer1,Mer2}). Technical adjustments are needed to make it work in the case of nodal reducible curves. 
\vspace{2mm}

We recall that if $F$ is a $\w$-semistable depth one sheaf with $\wrank(F)= r$  and $\wdeg(F) = d'$ big enough, then $F$ is a quotient of a trivial sheaf on $C$ of rank $N= d'+ r(1 -p_a(C))$ (see \cite[Proposition 16, Chapter 7]{Ses}). Let $Q$ be the Quot scheme parametrizing quotients of $\cO_C^{N}$ with fixed Hilbert polynomial $p$ and fixed multirank $\ur$ and let denote by  $\cF$ the universal family of quotients.
Let $R^{s} \subset Q$ be the subscheme parametrizing quotients $q:\cO_C^{N} \to \cF_q$  where $\cF_q$  is  a $\w$-stable  depth one  sheaf and such that $H^0(q):\mathbb C^N \to H^0(\cF_q)$ is an isomorphism. 
 We denote by $\cF^s$ the restriction of $\cF$ to $R^s \times C$, it is a coherent sheaf on $R^s \times C$  which is flat over $R^s$.  As usual we denote by $p_i$, $i= 1,2$, the projections of $R^s \times C$ onto factors. 
 By \cite[Theorem 19, Chapter 7]{Ses}, the moduli space $\Uw^s(\ur,d)$ is a good quotient of $R^s$ for the action of $\SL(N)$; so we have a proper morphism $\pi' \colon R^s \to \Uw^s(\ur,d')$.
We recall that we have an isomorphism 
$$\Uw^s(\ur,d)  \to \Uw^s(\ur,d')$$
by tensoring any sheaf $F$ with a line bundle  $L$ on $C$  as long as the   restrictions $L_i$ on the component $C_i$ satisfy  the condition
$$\deg(L_i)w_j = \deg(L_j)w_i \qquad \forall i,j\in \{1,\dots,\gamma\}.$$
Hence we  can consider  the proper morphism $\pi:R^s\to \Uw^s(\ur,d)$ defined by composition.
To give a scheme structure to Brill-Noether loci  we proceed as in the smooth case: we will see $\Bw(\ur,d,k)$  as the image by $\pi$ of a degeneracy locus $R(\ur,d,k)  \subset R^s$ of a suitable map between vector bundles. 
We assume that $R^s$ is irreducible, in general it is enough  to consider each irreducible component.  
We choose an effective  divisor $D$ on the curve $C$ satisfying the following   conditions:  any  $x \in \Supp(D)$ is a smooth point of $C$ and $\wdeg( \OO_C(D))= a >>0$.  Then $p_2^*(\OO_C(D)) \simeq \OO_{R^s \times C}(R^s \times D)$.
Let's consider the  following sheaves on $R^s$: 
$$G_1 = {p_1}_{*}(\cF^s \otimes p_2^*(\OO_C(D)) \quad  G_2 = {p_1}_{*}({\cF^s \otimes p_2^*(\OO_C(D)}_{\vert R^s \times D})).$$
If $a$ is sufficiently big, by Grauert's Theorem, $G_1$ and $G_2$ are vector bundles on $R^s$ whose fibers are
$$(G_1)_{q}  \simeq  H^0((\cF^s)_q \otimes \cO_C(D)) \quad\mbox{ and }\quad  
(G_2)_{q} \simeq  H^0( {(\cF^s)_q \otimes \cO_C(D)}_{\vert D})$$
respectively.

For any $q \in R^s$, $(\cF^s)_q$ is a depth one sheaf on $C$ which is $\w$-stable and it fit into  the following exact sequence:
$$ 0 \to (\cF^s)_q \to (\cF^s)_q \otimes \cO_C(D) \to (\cF^s)_q \otimes \cO_C(D)_{\vert D} \to 0.$$
We have a map  of vector bundles 
$\Phi \colon G_1 \to G_2$, such that  for any $q \in R^s$  the map on the fibers $\Phi_q$ fit into the following exact sequence
$$ 0 \to H^0((\cF^s)_q) \to H^0((\cF^s)_q \otimes O_C(D))\stackrel{\Phi_q}\to H^0((\cF^s)_q \otimes O_C(D)_{\vert D}) \to H^1((\cF^s)_q) \to 0.$$
Let $R(\ur,d,k)$ be the degeneracy locus in $R^s$ of points $q$ such that $\rk(\Phi_q) \leq h^0(\cF_q \otimes \cO_C(D)) -k$. If $R(\ur,d,k)$ is not empty,  then, by  \cite[Chapter 2, page 83]{ACG}, every irreducible component has codimension at most 
$$(\rk(G_1)-(h^0(\cF_q \otimes \cO_C(D))-k))(\rk(G_2)-(h^0(\cF_q \otimes \cO_C(D))-k))=k(k-d+r(p_a(C)-1)).$$

As $R(\ur,d,k)$ is a closed and $\SL(N)$-invariant subscheme of $R^s$ and  $\pi$ is a good quotient, then $\Bw(\ur,d,k) = \pi(R(\ur,d,k))$ is a closed subscheme of $\Uw^s(\ur,d)$. Moreover, codimension is preserved as $R(\ur,d,k)$ is contained in the $\SL(N)$-stable locus.   So we can conclude that if $\Bw(\ur,d,k)$  is not empty  its  codimension   in $\Uw^s(\ur,d)$ is at most $k(k-d+r(p_a(C)-1))$.
\end{proof}

\begin{remark} The same construction allows  us to give a scheme structure to the Brill-Noether loci $\Btw(\ur,d,k) = \{ [F] \in \Uw(\ur,d) \,\vert\, h^0(gr(F)) \geq k \}$. Actually, as in the smooth case, we do not have any information about its  codimension. 
\end{remark}


Let $r \in \N$, in the sequel we will consider $\w$-semistable depth one sheaves on $C$ having rank $r$ an any irreducible component, i.e. with multirank 
$r\cdot \underline{1}=(r,r,\dots,r)$. If $E$ is such a sheaf, we have that 
$\wrank(E)=r$ and $d=\wdeg(E)=\chi(E)-r\chi(\cO_C)$,  so $d$ is an integer.
For any $d \in \N$, the moduli space $\Uw^s(r\cdot \underline{1},d)$ has been described in \cite{Tei91,Tei95}: it is   reducible and connected, each irreducible component has dimension $r^2(p_a(C)-1) +1$ and the 
general element is a $\w$-stable locally free sheaf whose restrictions to  irreducible components are stable too.
\medskip

As in the smooth case, we can define the Brill-Noether number
\begin{equation}
\label{BrillNoetherNumber}
\beta_C(r,d,k)= r^2(p_a(C)-1)+1-k(k-d+r(p_a(C)-1)),
\end{equation}
which is an integer under the above assumption.
\medskip

Assume that $E \in \Bw(r\cdot \underline{1},d,k)$ is a $\w$-stable locally free sheaf. 
Then, the Zariski tangent space $T_{E}(\Bw(r\cdot \underline{1},d,k))$  can be described as in the smooth case (see \cite{MerP}) in the following proposition.

\begin{proposition}
\label{PROP:Petri}
\hfill\par
\begin{enumerate}[(a)]
\item{} If $\Bw(r\cdot \underline{1},d,k) \not= \emptyset$ and $\Bw(r\cdot \underline{1},d,k) \not= \Uw^s(r\cdot \underline{1},d)$, then any irreducible component has dimension  at least $\beta_C(r,d,k)$.
\item{} Let $[E] \in \Bw(r\cdot \underline{1},d,k) \setminus \Bw(r\cdot \underline{1},d,k+1)$ be a locally free sheaf.
The Zariski tangent space  $T_E(\Bw(r\cdot \underline{1},d,k))$ is the annihilator of   the image of the Petri map:
$$\mu_E  \colon H^0(E) \otimes H^0(E^* \otimes {\omega}_C) \to  H^0(E \otimes E^* \otimes {\omega}_C);$$
 $\Bw(r\cdot \underline{1},d,k)$ is smooth of dimension $\beta_C(r,d,k)$ at $E$ if and only if $\mu_E$ is injective. 
\end{enumerate}
\end{proposition}
\begin{proof}
(a) follows by Proposition \ref{PROP:scheme} since each  irreducible component of $\Uw^s(r\cdot \underline{1},d)$ has dimension $r^2(p_a(C)-1) +1$. \hfill\par
(b) Since $E$ is $\w$-stable and locally free, the moduli space $\Uw^s(r\cdot \underline{1},d)$ is smooth at $[E]$ and the tangent space $T_{[E]}\Uw^s(r\cdot \underline{1},d)$ can be identified with $\Ext^1(E,E) \simeq H^1(C,E \otimes E^*)$, (see \cite[Corollary 4.5.2]{HL}). Note that if $E$ is not locally free then $[E]\in \Uw^s(r\cdot \underline{1},d)$ is a singular point by Lemma \ref{LEM:EXT1}. Let $[E] \in \Bw(r\cdot \underline{1},d,k) \setminus \Bw(r\cdot \underline{1},d,k+1)$ with $E$ locally free.
As in the smooth case (see \cite{MerP}), we can identify the Zariski tangent space  $T_{[E]}(\Bw(r\cdot\underline{1},d,k))$ as the kernel of the map
$$c \colon H^1(\cHom(E,E)) \to \Hom(H^0(E),H^1(E)),$$
which, in terms of cocycles and Cech cohomology can be described as the map
sending $(\phi_{ij}) \mapsto [ s \mapsto \phi_{ij}(s)]$.
Since $E$ is locally free the dual map of $c$ is the Petri map.
 \end{proof}

There is a strong relation between coherent systems and Brill-Noether loci, as the next proposition shows.

\begin{proposition}
\label{PROP:COH}
Let $(C,\w)$ be a polarized nodal curve. Let $0<\alpha_1<\cdots < \alpha_L$ be the critical values for coherent systems of multitype $(\ur,d,k)$. Then
\begin{enumerate}[(a)]
\item If $(E,V)$ is $\wa$-stable for $\alpha\in (0,\alpha_1)$, then $E$ is $\w$-semistable;
\item if $E$ is $\w$-stable and $h^0(E)\geq k$, then for all $V\subseteq H^0(E)$ with $\dim V=k$, $(E,V)$ is $\wa$-stable for $\alpha\in (0,\alpha_1)$;
\item let $(E,V)\in \GwL(\ur,d,k)$,  then $E$ is $\w$-unstable if and only if $(E,V)$ is $\wa$-unstable for $\alpha<\alpha_1$.
\end{enumerate}
\end{proposition}

\begin{proof}
The proof of $(a)$ and $(b)$, as in the smooth case (see \cite{RV}), follows directly from the definitions of $\w$-(semi)stability and of $\wa$-(semi)stability. The proof for $(c)$  works as in the smooth case (see \cite{BG02}).
\end{proof}

A simple but relevant consequence of the above proposition is the following.

\begin{proposition}
Let $(C,\w)$ be a polarized nodal curve with $\w$ good. If $\Bw( \underline{r},d,k)\neq \emptyset$  then $d \geq 0$. Moreover, if $1\leq k < \sum_{i=1}^\gamma w_ir_i$ then $d>0$.
\end{proposition}
\begin{proof}
Let $[E]\in \Bw(\underline{r},d,k)$. Then $E$ is $\w$-stable and $h^0(E)\geq k$. Let $V$ be any subspace of $H^0(E)$ of dimension $k$. Consider the evaluation map $\ev_V:V\otimes \cO_C\to E$.
If it is surjective, then $E$ is globally generated, so $\wdeg(E) \geq 0$ by \cite[Theorem 2.9(b)]{BFPol}.
On the contrary, let $F$ be the image of $\ev_V$. Then, $F$ is a globally generated sheaf of depth one and it is a subsheaf of $E$ which is $\w$-stable. Hence we have
$$0\leq \wdeg(F)/\wrank(F)<\wdeg(E)/\wrank(E)$$ 
which implies $\wdeg(E)>0$. Finally, if $k < \wrank(E)$, the evaluation map $\ev_V$ cannot be surjective. 
\end{proof}

\section{Brill-Noether loci for sheaves with small slope}
\label{SEC:smallslope}
Let $(C, \w)$ be a polarized nodal curve.
In this paper we are interested in studying Brill-Nother loci for depth one sheaves having rank $r$ on all irreducible components of $C$.  They will include  the corresponding loci for vector bundles. 
\medskip

We recall that in the smooth case (see \cite{BGN97}) and in the irreducible nodal case (see \cite{Bho07}), all the elements of Brill-Noether loci for small slope (i.e. $0 \leq d \leq r$) are defined by BGN extensions. This is not true anymore when $C$ is a reducible nodal curve, as it will be  shown in  Example \ref{EX:noBGN}. 
However, we will prove that this actually holds when we consider locally free sheaves. This is  stated in the following Theorem,  which can be seen as a partial generalization of Theorems $A+B$ of \cite{BGN97}. 


\begin{theorem}
\label{THM:genTHMb}
Let $(C,\w)$ be a polarized nodal curve with $\w$ good.  Let $d, r, k \in \mathbb N$ with $r \geq 2$, $k \geq 1$ and $d \geq 0$. Let $E$ be a locally free sheaf in $\Bw( r\cdot \underline{1},d,k)$ which satisfies at least one of the following two conditions:
\begin{enumerate}[(a)]
    \item $ 0 \leq d \leq r$;
    \item the restriction $E_i$ is stable and $ 0 \leq \deg(E_i)\leq r$ for all $i=1,\dots,\gamma$.
\end{enumerate}
Then 
$$d>0,\qquad k<r\leq d+(r-k)p_a(C)$$ 
and $E$ is obtained as a BGN extension of a locally free sheaf of rank $r-k$.
\end{theorem}
\begin{proof}
Let $E \in \Bw( r\cdot \underline{1},d,k)$ be a $\w$-stable locally free  sheaf. Let $V \subseteq H^0(E)$ be a subspace of dimension $k$.  
We claim that the evaluation map  $\ev_V \colon V \otimes \OO_C \to E$ is  an injective map of vector bundles. 
Since the map induced on the fibers at the point $x \in C$ is the map sending $(s,x) \mapsto s(x)$, it is enough to verify that $s(x)\not= 0 $ for any non-zero $s\in V$ and for any  $x \in C$.
Let $s \in V$ be a non-zero section. We consider the map: $${\ev_s} \colon = {\ev_V}|_{\langle s\rangle\otimes \cO_C}: \langle s\rangle \otimes  \cO_C \to E,$$
and let $\mathcal L$ be its image. It is a depth one subsheaf of $E$ which is globally generated by construction. We denote by $\cL_i$ its restriction to $C_i$ modulo torsion. If $\cL_i$ is not the zero sheaf, then, by \cite[Lemma 3.3]{BFCoh}, we have the following commutative diagramm
$$
\xymatrix{
\langle s\rangle \otimes \OO_C \ar@{->>}[r]^-{ev_s} \ar@{->>}[d] & {\mathcal L} \ar@{->>}[d]  \\
\langle s_i\rangle \otimes \OO_{C_i} \ar@{->>}[r]^-{ev_{V,C_i}} & {\mathcal L}_i}
$$
from which we deduce that $\cL_i$ is a line bundle generated by $\langle s_i\rangle$ where $s_i=s|_{C_i}$. In this case, we have $\deg(\cL_i) \geq 0$ with $\deg(\cL_i)=0$ if and only if $\cL_i=\cO_{C_i}$. We prove  that if $\cL_i\neq 0$ then  we have $s_i(x) \neq0$ for any $x\in C_i$, i.e. $\cL_i= \OO_{C_i}$. On the contrary, we would have $\deg(\cL_i) \geq 1$. We claim this can not happen.
Indeed, if we are in case $(a)$, since $\w$ is good, by Equation \eqref{EQ:DELTA}, we would have
$$\wdeg(\cL) \geq \sum_{i=1}^{\gamma}\deg(\mathcal L_i) \geq 1  \ \ \text{and} \ \ \wrank(\mathcal L) \leq \sum_{i=1}^{\gamma}w_i \leq 1,$$
so $\mu_{\w}(\mathcal L) \geq 1$. This is impossible since $E$ is $\w$-stable with slope $\mu_{\w}(E) \leq 1$. Instead, if we are in case $(b)$, $\cL_i$ is a subsheaf of $E_i$ with $\mu(\cL_i)\geq 1$, which contradicts the assumption on the stability of $E_i$.
\medskip

Finally, we prove that ${\mathcal L}_i \simeq \OO_{C_i}$ for any $i$. 
Assume, by contradiction, that the restriction of $s$ to at least one   component of $C$ is identically zero. Then we can find two different components $C_i$ and $C_j$ such that $p\in C_i\cap C_j$, $s_i\not\equiv 0$, $s_j\equiv 0$. Then, since $E$ is locally free and $s$ is a global section of $E$, we would have $s_i(p)=s_j(p)=0$. But we have shown above that $s_i$ cannot have zeros since it is a section of $\cL_i=\cO_{C_i}$.
\medskip

We have shown that $\ev_V$ is an injective map of vector bundles. The $\w$-stability of $E$ implies that $\ev_V$is not an isomorphism and that $d= \wdeg(E)>0$. Moreover, we have an exact sequence
$$0\to V\otimes \cO_C\to E\to F\to 0$$
with $F$ locally free of rank $r-k\geq 1$. In particular we have $k<r$ as claimed. By Proposition \ref{PROP:COH}$(b)$, $(E,V)$ is an $\wa$-stable coherent system for $\alpha$ small enough. Then, by \cite[Lemma 3.12]{BFBGN}, we have that the above exact sequence is a BGN extension. By \cite[Proposition 2.3]{BFBGN} and \cite[Lemma 1.8]{BFBGN} we have $h^0(F^*) = 0$ and $k\leq h^1(F^*)=d+(r-k)(p_a(C)-1)$. This is equivalent to $r\leq d+(r-k)p_a(C)$.
\end{proof}

\begin{example}
\label{EX:noBGN}
Let $C_1$ and $C_2$ be  smooth curves  of genus $g_1$ and $g_2$, respectively, such that $3\leq g_2<g_1$. Let $C$ be the nodal curve obtained by gluing $C_1$ and $C_2$ at the points $q_1$ and $q_2$; we denote by $p$ the node of $C$. Under these assumptions we have the following facts:
\begin{itemize}
\item the moduli space $\cU_{C_i}(2,1)$ has dimension $4g_i-3$;
\item the Brill-Noether locus $\cB_{C_i}(2,1,1)$ is non-empty, it is irreducible and smooth and has dimension $2g_i-1$, so it is a proper subvariety of $\cU_{C_i}^s(2,1)$. This is a consequence of \cite[Theorems A+B]{BGN97};
\item the locus 
$$Y_i:=\{F\in \cU_{C_i}^s(2,1)\,|\, \exists\, L\in \Pic^0(C_i) \mbox{ s.t. } h^0(F\otimes L)\geq 1\}$$
is a proper closed subscheme of $\cU_{C_i}^s(2,1)$. Indeed, one can show that $Y_i$ has dimension at most $\dim(\cB_{C_i}(2,1,1))+\dim(\Pic^0(C_i))$.
\end{itemize}
We consider $E_1 \in \cB_{C_1}(2,1,1)$ and $E_2\in \cU_{C_2}^s(2,1)\setminus (Y_2\cup \cB_{C_2}(2,1,1))$. Then\footnote{We recall that a vector bundle $F$
on a smooth curve is $(m,n)$-semistable (respectively $(m,n)$-stable) if, for any  subsheaf $G$ of $F$, we have $\frac{\deg(G)+m}{\rank(G)}\leq \frac{\deg(F)+m-n}{\rank(F)}$ (respectively $<$). For details see, \cite{NR}.} $E_1$ is $(0,1)$-semistable and $E_2$ is $(0,2)$-stable. Since $\cB_{C_1}(2,1,1)$ is smooth, we have $h^0(E_1)=1$ and $E_1$ is given by a BGN extension
$$0\to \cO_{C_1}\to E_1\to L\to 0$$
where $L\in\Pic^1(C_1)$. Let $s$ be a generator of $H^0(E_1)$; notice that $s$ does not have any zero by construction. We consider a linear map $\sigma$ between the fibers of $E_1$ and $E_2$ at the points $q_1$ and $q_2$, respectively, such that $\ker(\sigma)=\langle s(q_1)\rangle$. Then, following \cite{BFVec}, we can construct a depth one sheaf $E$ on $C$ which fits into the exact sequence 
\begin{equation}
\label{EQ:SESDEFE}
0 \to E \to E_1 \oplus E_2 \to \bC_p^2 \to 0.
\end{equation}
This, roughly speaking, can be done by gluing the fibers of $E_1 $ and $E_2$ at the points $q_1$ and $q_2$ according to $\sigma$ (see \cite[Definition 3.1]{BFVec}).
By \cite[Proposition 3.2]{BFVec} we have that $E$ is a sheaf with multirank $(2,2)$, $\chi(E)=4-2g_1-2g_2$ and it is not locally free. We fix the canonical polarization $\underline{\eta}$ on the curve $C$ and we observe that it is good as $C$ is of compact type by \cite[Proposition 2.8]{BFPol}). 
Then, we have $\deg_{\underline{\eta}}(E)=2$ and $\mu_{\underline{\eta}}(E)=1$. One can  show that $\underline{\eta}$ satisfies the stability conditions \cite[Equation (3.3)]{BFVec} since we are assuming $g_1>g_2$. Then, \cite[Proposition 3.6]{BFVec} guarantees that $E$ is $\underline{\eta}$-stable. Finally, from the exact sequence \eqref{EQ:SESDEFE}, as a consequence of our choice of $\sigma$, we have that $H^0(E)\simeq   H^0(E_1)\oplus H^0(E_2)\simeq H^0(E_1)$. 
So we can conclude that $E \in \cB_{(C,\underline{\eta})}(2\cdot \underline{1},2,1)$ and any global section of $E$ vanishes on $C_2$: this implies that $\ev:H^0(E)\otimes \cO_C\to E$ is not injective so $E$ can not be obtained as BGN extension.
\end{example}

\subsection{Constructing irreducible components of Brill-Noether loci via BGN extensions}
We would like to describe irreducible components of Brill-Noether loci of locally free sheaves with  small slopes,   using BGN extensions defining $\w$-stable depth one sheaves.
\medskip

From now on, we will assume that $(C,\w)$ is a polarized nodal curve of compact type with $\w$ good and with $\gamma$ smooth irreducible components of genus $g_i \geq 2$. We give an ordering $\{C_1, \dots, C_{\gamma}\}$ for the irreducible components of $C$ and we define the family of subcurves $\{A_j \}_{j= 1,\dots, \gamma -1}$ according to Lemma \ref{LEM:ORDER}.
Let $\Uw(s \cdot \underline{1},d)$ be the moduli space of $\w$-semistable depth one sheaves  with multirank $s \cdot \underline{1}$ and $\w$-degree $d$. 
The following result summarizes some technical conditions on $\w$-stability:

\begin{lemma}
\label{LEM:STARCONDITIONS}
In the above hypothesis, we have the following properties.
\begin{enumerate}[(a)]
\item{} Let $F$ be a locally free sheaf of rank $s$ and degree $d$ whose restrictions $F_i$ are stable with degree $d_i$. If  the  following conditions hold:
\begin{equation}
\label{EQ:STARJ}
    (\star)_j\,:\qquad \wrank(\cO_{A_j})d-s\Delta_{\w}(\cO_{A_j})<\sum_{\substack{C_i\subseteq A_j}}d_i<\wrank(\cO_{A_j})d+ s(1-\Delta_{\w}(\cO_{A_j}))
\end{equation}
for $j=1,\dots, \gamma-1$, then $F$ is $\w$-stable. Conversely, a  general element of $\Uw(s\cdot \underline{1},d)$ is locally free, has stable restrictions of degree $d_i$ satisfying the above conditions.

\item Irreducible components of the moduli space 
$\Uw(s \cdot {\underline 1},d)$ correspond to $\gamma$-uples $(d_1,\dots,d_{\gamma})\in \bZ^{\gamma}$ with $\sum_{i=1}^{\gamma}d_i= d$ and which satisfy condition $(\star)_j$ for $j=1,\dots, \gamma-1$.

\item If $F$ satisfies condition $(\star)_j$ for $j=1,\dots, \gamma-1$ for a polarization $\w$, then the same holds for any polarization $\w'$ in a neighborhood of $\w$. If $\w$ is good then $\w'$ is good too in a suitable neighborhood.
\end{enumerate}
\end{lemma}

\begin{proof}
$(a)$ and $(b)$ are the main results of \cite{Tei91}.
We only need to prove that the stability conditions can be expressed as in Equation \eqref{EQ:STARJ} using the language and the notations introduced in \cite{BFPol}. The conditions of \cite{Tei91} are
 $$\wrank(\cO_{A_j})\chi(F)-\sum_{\substack{C_i\subseteq A_j \\ i\ne j}}\chi(F_i) + s(a_j-1) < \chi(F_j)  <
\wrank(\cO_{A_i})\chi(F)-\sum_{\substack{C_i\subseteq A_j \\ i\ne j}}\chi(F_i) + sa_j,$$
where $\{ A_j \}_{j=1,\dots \gamma -1}$ are subcurves that satisfy the requests of Lemma \ref{LEM:ORDER} and $a_j$ is the number of irreducible components of $A_j$.
Using the equalities
$\chi(F)= d +s(1-p_a(C))$ and $\chi(F_i) = d_i + s(1 -g_i)$ we obtain:
$$\wrank(\cO_{A_j})d-s\left[1-\sum_{C_i \subseteq A_J} g_i - \wrank(\cO_{A_j})(1-p_a(C))\right] < \sum_{\substack{C_i\subseteq A_j}} d_i  <$$
$$ < \wrank(\cO_{A_i})d + s\left[\sum_{C_i\subseteq A_j}g_i + \wrank(\cO_{A_j})(1-p_a(C))\right].$$
By the definition of the $\Delta_{\w}$ function (see \eqref{EQ:DELTA}) we have:
$$\Delta_{\w}(\cO_{A_j})= \wdeg( \cO_{A_j})= \chi(\cO_{A_j}) - \wrank(\cO_{A_j})\chi(\cO_{C})=
1 - \sum_{C_i \subseteq A_j}g_i- \wrank(\cO_{A_j})(1 - p_a(C)),$$
which implies \eqref{EQ:STARJ}.
\hfill\par

$(c)$ Let $\w'$ be a polarization and set 
$\underline{\epsilon}=(\epsilon_1,\dots,\epsilon_{\gamma})=\w'- \w$.
Note that, by construction, $\sum_{i=1}^{\gamma}\epsilon_i= 0$.
Assume that $(\star)_j$ holds for $\w$ for any $j=1,\dots \gamma -1$, we prove that if   $\underline{\epsilon}$ is sufficiently small, then $(\star)_j$ holds for $\w'$ for any $j=1,\dots \gamma -1$.
In fact we have
$$ {\rm rk}_{\w'}(\cO_{A_j})= \wrank( \cO_{A_j})+ \sum_{C_i \subseteq A_j} \epsilon_i \quad \mbox{ and }\quad \Delta_{\w'}(\cO_{A_j}) =\Delta_{\w}(\cO_{A_j}) - \chi(\cO_C)\sum_{C_i \subseteq A_j} \epsilon_i.$$
Condition $(\star)_j$ for $\w'$  is the following:
$$
\wrank(\cO_{A_j})d-s\Delta_{\w}(\cO_{A_j}) + (d +s \chi(\cO_C)) \sum_{C_i \subseteq A_j}\epsilon_i  <\sum_{\substack{C_i\subseteq A_j}}d_i<
$$
$$ < \wrank(\cO_{A_j})d+ s(1-\Delta_{\w}(\cO_{A_j})) + (d + s \chi(\cO_C)) \sum_{C_i \subseteq A_j}\epsilon_i, 
$$
hence it holds for $\epsilon_i$ sufficiently small. 

Finally,  as being good is an open condition (\cite[Corollary 3.15]{BFPol}), if $||\underline{\epsilon}||$ is small enough we have that $\w'$ is a good polarization too. 
\vspace{2mm}
\end{proof}

In the sequel, we will denote by $X_{d_1,\dots,d_{\gamma}}$ the irreducible component of $\Uw(s \cdot {\underline 1},d)$ corresponding to the $\gamma$-uple $(d_1,\dots,d_{\gamma})$ according to the above Lemma.

\begin{remark}
\label{REM:canonicalpol}
\rm
Let $\underline{\eta}$ be the canonical polarization on $C$, i.e. the polarization induced by $\omega_C$. As $C$ is a stable curve with $p_a(C)\geq 2$, $\underline{\eta}$ is good (see \cite[2.8]{BFPol}). 
We claim that the condition $(\star)_j$ for the canonical polarization can be written as follows:
\begin{equation}
\label{EQ:STARJcanonical}
    {\rm rk}_{\underline{\eta}}(\cO_{A_j})d-s/2<\sum_{\substack{C_i\subseteq A_j}}d_i<{\rm rk}_{\underline{\eta}}(\cO_{A_j})d+s/2.
\end{equation}
Indeed, by \cite[2.8]{BFPol}, it follows that for any subcurve $B$ of $C$ we have $\Delta_{\underline{\eta}}(\cO_B)=B\cdot B^c/2$.
Since the curves $A_j$ satisfy the requests of  Lemma \ref{LEM:ORDER}, we have $\Delta_{\underline{\eta}}(\cO_{A_j})=1/2$ and this gives the claim.
\end{remark}

Let $d >0$ and $0<k<r$ be integers. 
We recall that in Subsection \ref{Subsec:depthone} we have denoted by $\GwL(r\cdot\underline{1},d,k)$ the terminal moduli space for coherent systems of multitype $(r\cdot{\underline 1},d,k)$ on $C$. 
By \cite[Theorem 5.1]{BFBGN} each non-empty irreducible component $Y_{d_1,\dots,d_{\gamma}}$ of this space has dimension equal to the Brill-Noether number $\beta_{C}(r,d,k)$ (see \eqref{BrillNoetherNumber}) and its general element is a  pair $(E,V)$  with $E$ locally free and  $\deg(E_i)= d_i$.
The following proposition gives sufficient conditions for the $\w$-stability of $E$.

\begin{proposition}
\label{PROP:Estab}
Let $(C,\w)$ be a polarized nodal curve of compact type with $\w$ good. Let $r, d$ and $k$ be as above and consider a non-empty irreducible component $Y_{d_1,\dots,d_{\gamma}} \subset \GwL(r\cdot\underline{1},d,k)$ with $0 < d_i \leq r$ for any $i=1,\dots \gamma$. Assume moreover that
\begin{equation}
\label{INEQ:BGN_Ci}
k \leq \frac{d_i + r(g_i-1)}{g_i} \qquad \mbox{ for all } i = 1,\dots,\gamma.
\end{equation}
Then $E$ is  $\w$-stable for a general element $(E,V) \in Y_{d_1,\dots,d_{\gamma}}$. 
\end{proposition} 

\begin{proof}
Let $Y_{d_1,\dots,d_{\gamma}}$  be a non-empty irreducible component of $\GwL(r\cdot\underline{1},d,k)$.
By \cite[Theorem 5.1(b)]{BFBGN}, there exists an irreducible component 
$X_{d_1,\dots,d_{\gamma}}$ of the moduli space $\U_{(C,\w)}((r-k)\cdot \underline{1},d)$
and a dominant morphism  
$$\psi \colon Y_{d_1,\dots,d_{\gamma}} \to X_{d_1,\dots, d_\gamma}\qquad (E,V)\mapsto \coker(ev_V),$$
where $ev_V$ is the evaluation map of global sections of $V$. The fiber over a $\w$-stable sheaf $F$ is isomorphic to $\Gr(k,H^1(F^*))$. More precisely, in
\cite[Proposition 3.3]{BFBGN} it is shown that $\Gr(k,H^1(F^*))$ parametrises BGN extensions of $F$ of type $(r,d,k)$ (see Subsection \ref{Subsec:depthone}). The isomorphism takes $\underline{e}\in \Gr(k,H^1(F^*))$ to the coherent system $(E,V)$ induced by the BGN extension
\begin{equation}
\label{SES:BGN}  
\underline{e}:\qquad 0\to V\otimes \cO_{C}\to E\to F\to 0
\end{equation}
defined by $\underline{e}$.

By Lemma \ref{LEM:STARCONDITIONS}(a), a  general $F\in X_{d_1,\dots,d_{\gamma}}$ is locally free, $\w$-stable, each restriction $F_i$ is stable of degree $d_i$ and conditions $(\star_j)$ holds, for $j=1,\dots,\gamma-1$. 

\vspace{2mm}

{\bf Claim (a)}: for a general $F\in X_{d_1,\dots,d_{\gamma}}$ and for any $(E,V)\in \psi^{-1}(F)$, we have that $E$ is locally free and satisfies $(\star)_j$ for all $j=1,\dots \gamma-1$.
\vspace{2mm}

Since $F$ is locally free we have that $E$ is locally free too (by \cite[Proposition 3.3(a)]{BFBGN}) and, by tensoring by $\cO_{C_i}$ the exact sequence \eqref{SES:BGN}, we get again an exact sequence. The latter yields $\deg(E_i)=\deg(F_i)= d_i$. Since $F$ satisfies Condition $(\star_j)$, 
we have
$$
\sum_{C_i \subseteq A_j} d_i > 
\wrank(\cO_{A_j})d -(r-k)(\Delta_{\w}(\cO_{A_j}) = \wrank(\cO_{A_j})d -r\Delta_{\w}(\cO_{A_j}) + k \Delta_{\w}(\cO_{A_j}),$$
$$
\sum_{C_i \subseteq A_j} d_i < \wrank(\cO_{A_j})d +(r-k)(1 - \Delta_{\w}(\cO_{A_j})) = 
\wrank(\cO_{A_j})d +r (1 - \Delta_{\w}(\cO_{A_j})) +k (\Delta_{\w}(\cO_{A_j}) -1).$$
Now, we recall that  since $\w$ is a good polarization, then $\cO_C$ is $\w$-stable (by \cite[Theorem 2.9]{BFPol}).  By Lemma \ref{LEM:ORDER}(c), the intersection $A_j \cap A_j^c$  is a single node, so we have $0<\Delta_{\w} (\cO_{A_j})<1$ by \cite[Proposition 2.12]{BFPol}. This and the above inequalities imply that $E$ satisfies Condition $(\star)_j$, for $j=1,\dots, \gamma -1.$  
\medskip

{\bf Claim (b)}: for a general $F \in X_{d_1,\dots d_{\gamma}}$ and general $(E,V) \in \psi^{-1}(F)$, the restrictions $E_i$ are stable.
\hfill\par

Since $F$ is general we can assume that it is $\w$-stable. By Conditions \eqref{INEQ:BGN_Ci} and by \cite[Corollary 5.5]{BFBGN}, for a general $(E,V) \in \psi^{-1}(F)$ the restriction $(E_i,V_i)$ is an element of the moduli space ${\mathcal G}_{C_i,L}(r,d_i,k)$. 
Recall that, since $0 <d_i \leq  r$, elements of $\mathcal{G}_{C_i,L}(r,d_i,k)$ correspond to BGN extensions of semistable locally free sheaf. 
More precisely, there exists a dominant morphism 
$$\psi_i:\mathcal{G}_{C_i,L}(r,d_i,k)\to {\mathcal U}_{C_i}(r-k,d_i)$$
whose fiber over a stable $M$ is $\psi_i^{-1}(M)\simeq\Gr(k,H^1(M^*))$. Moreover, for $(G,W)$
general in $\mathcal{G}_{C_i,L}(r,d_i,k)$, we have that $G$ is stable. These assertions follows from \cite{BGN97} and \cite{BG02}. Then, in order to prove the claim, it is enough to show that $(E_i,V_i)$ is general in ${\mathcal G}_{C_i,L}(r,d_i,k)$.
\medskip

A general stable $F \in X_{d_1,\dots d_{\gamma}}$ is obtained, by the results of \cite{Tei91}, as follows: one first take, for all $i=1,\dots,\gamma$, a general $F_i\in \cU_{C_i}^s(r,d_i)$ and, for each node $p\in C_i\cap C_j$, one chose an isomorphism between the fibers of $F_i$ and $F_j$ at $p$. The sheaf $F$ is obtained by gluing $F_1, \dots F_{\gamma}$ along these fibers according to these choices. 

By \cite[Proposition 3.4]{BFBGN} we have a rational surjective map
$$
\xymatrix{
\Gr(k,H^1(F^*)) \simeq  \psi^{-1}(F) \ar@{-->}[r]^-{T_i} & \Gr(k,H^1(F_i^*))\simeq \psi_i^{-1}(F_i)
}$$
induced by restriction on $C_i$. This implies that for general $(E,V) \in \psi^{-1}(F)$, the restriction $(E_i,V_i)$ is defined by a BGN extension of $F_i$, i.e., 
$(E_i,V_i) \in \psi_{i}^{-1}(F_i)$. Let $U_i$ be the open dense subset of $\psi_i^{-1}(F_i)\simeq \Gr(k,H^1(F_i^*))$ corresponding to coherent systems $(E_i,V_i) \in {\mathcal G}_{C_i,L}(r,d_i,k)$ with $E_i$ stable. Then $\underline e  \in \bigcap_{i=1}^{\gamma}T_i^{-1}(U_i)$ corresponds to a coherent system $(E,V)$ with  $E_i$ stable for any $i=1,\dots \gamma$, as claimed.
\vspace{2mm}

Now we can conclude  the proof of the theorem. Let $F \in X_{d_1,\dots,d_\gamma}$ be a general $\w$-stable locally free sheaf and let $(E,V) $ be a general element in $\psi^{-1}(F)$. Then, by Claim (a), $E$ is locally free, it satisfies conditions $(\star)_j$ and, by Claim (b), its restrictions $E_i$ are stable. By Lemma \ref{LEM:STARCONDITIONS}$(a)$ it follows that $E$ is $\w$-stable. 
\end{proof}


We have now the second main result of this section.

\begin{theorem}
\label{THM:Main}
Let $(C,\w)$ be a polarized nodal curve of compact type with $\w$ good. Let $s,k\in \bN$ such that
$$\qquad k \leq 1 + s(g_i -1) \mbox{ for all } i=1,\dots, \gamma.$$ 
Assume that there exists a non-empty irreducible component $X_{d_1,\dots,d_\gamma}$ of $\Uw(s\cdot \underline{1},d)$ such that $0< d_i \leq s$ for $i=1,\dots, \gamma$. Then, if $r=s+k$, we have the following facts: 
\begin{itemize}
    \item there exists an irreducible component $Z$ of $\Bw(r\cdot \underline{1},d,k)$ with dimension $\beta_C(r,d,k)$;
    \item $Z$ is birational to a fibration over $X_{d_1,\dots,d_\gamma}$ in grassmannian varieties.
\end{itemize}
In particular, the Brill-Noether locus $\Bw(r\cdot \underline{1},d,k)$ is non-empty.
\end{theorem} 

\begin{proof}
First of all, we will prove that the general $F\in X_{d_1,\dots,d_{\gamma}}$ is a $\w$-stable locally free sheaf with $h^0(F) = 0$. Indeed, by Lemma  \ref{LEM:STARCONDITIONS} (a), $F$ is obtained by gluing general stable locally free sheaves $F_i$ of rank $s$ with $\deg(F_i)=d_i$. The relation between $F$ and its restrictions is given by the exact sequence (see \cite{Ses})
$$0\to F\to \bigoplus_i F_i\to T \to 0$$
where $T$ has support on the nodes of $C$ and the rank of $T$ at each node is exactly $s$.
By the results about Brill-Noether theory of higher-rank on smooth curves (see \cite[Theorem A]{BGN97}) a general $F_i \in{\cU}_{C_i}(s,d_i)$ has $h^0(F_i)=0$ since $d_i \leq  s$. Then, the above sequence  implies $h^0(F)=0$.
\vspace{2mm}

By the assumptions on $k$, we have
$k \leq 1 + s(g_i-1) \leq d_i + (r-k)(g_i-1), $ which implies 
$kg_i \leq d_i + r(g_i-1)$. Since $C$ is of compact type, this implies $kp_a(C)\leq d+r(p_a(C)-1)$. Then, by \cite[Theorem 5.1(b)]{BFBGN}, there exists an irreducible component $Y_{d_1,\dots,d_\gamma}$ of the moduli space  $\GwL(r\cdot\underline{1},d,k)$  and a dominant morphism $\psi \colon Y_{d_1,\dots,d_\gamma} \to X_{d_1,\dots,d_\gamma}$  whose fiber over $F$ is isomorphic to  $\Gr(k,H^1(F^*))$.
By Proposition \ref{PROP:Estab}, for a general coherent system $(E,V)\in Y_{d_1,\dots,d_\gamma}$ we have that $E$ is $\w$-stable.
Moreover, for a  general coherent system $(E,V)\in Y_{d_1,\dots,d_\gamma}$ we also have that
$h^0(E) = k$: this follows from the cohomological exact sequence induced by exact sequence \eqref{SES:BGN}
since $h^0(F)=0$ for $F$ general in $X_{d_1,\dots,d_\gamma}$. Then it is well defined, as rational map, the forgetfull map 
$$\xymatrix{ 
Y_{d_1,\dots,d_\gamma}\ar@{-->}[r]^-{f} & \Bw(r\cdot \underline{1},d,k) & (E,V)\ar@{|->}[r] & E
}$$
This proves that $\Bw(r\cdot \underline{1},d,k)$ is non-empty. 
\medskip

Let $Z$ be the image of $f$ (more precisely, the closure of the image of the domain of $f$). Consider $(E,V) \in Y_{d_1,\dots,d_\gamma}$ general. Since $h^0(E)=k$ , we have that $f^{-1}(E)  = \{(E,V) \}$. This implies that 
$$\dim Z = \dim Y_{d_1,\dots,d_\gamma}= \beta_C(r,d,k)$$ 
by \cite[Theorem 5.1(c)]{BFBGN}.
Finally,  a general element $E$ of $Z$ is a locally free sheaf, $\w$-stable, with $h^0(E) = k$ and the Petri map $\mu_E $ is injective,  see \cite[Proposition 2.13]{BFBGN}.
By Proposition \ref{PROP:Petri}, $\Bw(r\cdot \underline{1},d,k)$ is smooth at $E$ and it has dimension $\beta_C(r,d,k)$.
This implies that $Z$ is an irreducible component of $\Bw(r\cdot \underline{1},d,k)$.
\end{proof}

\begin{remark}
Let $Z$ be the irreducible component of $\Bw(r\cdot \underline{1},d,k)$ defined by $X_{d_1,\dots,d_{\gamma}}$ in Theorem \ref{THM:Main}. As consequence of the proof we have that if $E\in Z$ is locally free and $h^0(E)=k$ (i.e. $E\not \in \Bw(r\cdot \underline{1},d,k+1)$ ), then $Z$ is smooth at $E$.
\end{remark}

In light of Theorem \ref{THM:Main}, in order to obtain irreducible components of $\Bw(r\cdot \underline{1},d,k)$, it is worth to search for non-empty irreducible components $X_{d_1,\dots,d_{\gamma}} \subset \Uw((r-k)\cdot \underline{1},d)$ such that $0<d_i \leq r-k$ for all $i=1,\dots,\gamma$. This is equivalent to ask that all the restrictions $F_i$ of a locally free sheaf $F\in X_{d_1,\dots,d_{\gamma}}$ have slope $\mu_i$ in $(0,1]$. For brevity, in this case, we will say that $X_{d_1,\dots,d_{\gamma}}$ is a component with {\it small slopes}.


\section{Components with small slopes}
\label{SEC:SMALLSLOPE}

Let $C$ be nodal curve of compact type with $\gamma$ smooth irreducible components with genus $g_i\geq 2$ for all $i=1,\dots,\gamma$. Let $s,d\in \bN_{+}$. In this section we are looking for sufficient conditions for the existence of components with small slopes of $\Uw(s\cdot \underline{1},d)$ for a suitable good polarization $\w$. We recall that  $\underline{\eta}$  denotes the canonical polarization on $C$ (see Remark \ref{REM:canonicalpol}).
 
\begin{proposition}
\label{PROP:TEC}
Let $C$ be a nodal curve as above. Assume that one of the following conditions hold:
\begin{enumerate}[(a)]
    \item $\gamma\leq d\leq \frac{s}{2}+1$;
    \item $s/2+1< d\leq s, \, s/2\geq (\gamma-1)$ and there exists $i\in \{1,\dots, \gamma\}$ such that $\eta_id\geq s/2$;
    \item $s/2+1< d\leq s\gamma $ and $s/2\geq (\gamma-1)$, 
    let $n,m\in \bN$ be such that $d=n\gamma+m$ with $0\leq m < \gamma$
    and assume 
    \begin{equation}
    \label{EQ:EtaiD}
    n+1-\frac{s}{2(\gamma-1)}<\eta_i d< n+\frac{s}{2(\gamma-1)}
    \end{equation} for all but at most one index $i\in \{1,\dots, \gamma\}$.
\end{enumerate}
Then, there exists an open neighborhood $U$ of the canonical polarization such that for any $\w\in U$, $\Uw(s\cdot \underline{1},d)$ has a non-empty component $X_{d_1,\dots,d_{\gamma}}$ with small slopes.
\end{proposition}

\begin{proof}
First of all, we chose an ordering of the components of $C$ according to Lemma \ref{LEM:ORDER}. Hence, in case $(b)$ we can assume that $\eta_{\gamma}d\geq s/2$ whereas in case $(c)$ we can require that
condition \eqref{EQ:EtaiD} holds for all $i=1,\dots, \gamma-1$.
\medskip

Assume now that we are either in case $(a)$ or $(b)$.
We will show that (up to the above reordering) $X_{1,1,\dots,1,d-\gamma+1}$ is not empty for the canonical polarization $\underline{\eta}$, which is good as we have seen in Remark \ref{REM:canonicalpol}. By Lemma \ref{LEM:STARCONDITIONS}$(b)$, it will be enough to show that $d_1=\dots=d_{\gamma-1}=1$ satisfy conditions $(\star)_j$ for $j=1,\dots, \gamma-1$ and $0<d-\gamma+1\leq s$. Lemma \ref{LEM:STARCONDITIONS}$(c)$ will imply the result for a suitable neighborhood of $\underline{\eta}$.
\vspace{2mm}

We have that $d_1=\dots=d_{\gamma-1}=1$ satisfy condition $(\star)_j$ for $\underline{\eta}$ (i.e.  conditions  \eqref{EQ:STARJcanonical}) for any $j=1, \dots \gamma -1$ if and only if
\begin{equation}
\label{EQ:INEQfinal}
\rk_{\underline{\eta}}(\cO_{A_j})d-s/2<a_j<\rk_{\underline{\eta}}(\cO_{A_j})d+s/2
\end{equation}
where $a_j \geq 1$ is the number of components of $A_j$.
\vspace{2mm}

We prove now the inequality on the left of  \eqref{EQ:INEQfinal}. Under our assumption, we are able to prove the stronger inequality 
\begin{equation}
\label{EQ:INEQ1}
\rk_{\underline{\eta}}(\cO_{A_j})d-s/2<1
\end{equation}
for any $j=1,\dots \gamma-1$. 
\medskip

Indeed, in case $(a)$, Inequality \eqref{EQ:INEQ1} follows immediately from the assumption $d\leq s/2+1$. If we are in case $(b)$, assume, by contradiction, that there exists a curve $A_j$ such that $\rk_{\underline{\eta}}(\cO_{A_j})d\geq s/2+1$.
Then, since $C_\gamma$ is not a component of $A_j$ for all $j=1,\dots, \gamma-1$ by construction, we have
$$d =d\sum_{i=1}^{\gamma}\eta_i\geq \rk_{\underline{\eta}}(\cO_{A_j})d + \eta_{\gamma} d \geq s/2+1+ s/2\geq s+1$$
which is impossible since $d\leq s$. Finally, we have (in both cases)
\begin{equation}
1\leq a_j\leq \gamma-1\leq s/2<\rk_{\underline{\eta}}(\cO_{A_j})d+s/2.
\end{equation}
which implies the inequality on the right of \eqref{EQ:INEQfinal}.
\vspace{2mm}

One then concludes by observing that $d_{\gamma}=d-(\gamma-1)$ is such that $1\leq d_{\gamma}\leq s$ by assumption.
\medskip

Assume now that we are in case $(c)$. By the hypothesis on $d$, we have $n\leq s$ and $n\leq s-1$ if $m\neq 0$. 
Consider any component $X_{d_1,\dots,d_\gamma}$ where, for any $i=1,\dots \gamma$, $d_i$ is either equal to $n$ or to $n+1$. In particular we have $d_i=n+1$ for exactly $m$ values of $i$. We will prove $X_{d_1\dots,d_\gamma}$ is a non-empty component with small slopes for the canonical polarization. 
The conditions $(\star)_j$ for $\underline{\eta}$ can be written as follows:
\begin{equation}
\label{EQ:INEQfinal2}
\rk_{\underline{\eta}}(\cO_{A_j})d-s/2<\sum_{C_i\subseteq A_j}d_i<\rk_{\underline{\eta}}(\cO_{A_j})d+s/2.
\end{equation}
Notice that we have 
\begin{equation}
\label{EQ:Central}
na_j\leq \sum_{C_i\subseteq A_j}d_i \leq  (n+1)a_j
\end{equation}
where, as before, $a_j$ is the number of components of $A_j$.
Assumption \eqref{EQ:EtaiD} implies
$$\rk_{\underline{\eta}}(\cO_{A_j})d-s/2<na_j+\frac{s}{2(\gamma-1)}a_j-\frac{s}{2}=n a_j+\frac{s}{2}\left(\frac{a_j-(\gamma-1)}{\gamma-1}\right)\leq n a_j$$
and
$$\rk_{\underline{\eta}}(\cO_{A_j})d+s/2>(n+1)a_j-\frac{s}{2(\gamma-1)}a_j+\frac{s}{2}=(n+1) a_j+\frac{s}{2}\left(\frac{(\gamma-1)-a_j}{\gamma-1}\right)\geq (n+1) a_j$$
which implies  the desired conditions using  \eqref{EQ:Central}.
\end{proof}

\begin{remark}
The conditions $(b)$ $(c)$ gives constraints to the geometric configuration of the curve $C$. For example, in case $(b)$, roughly, there is a component with  very high genus or with a lot of nodes on it. More precisely, one has that there exists a unique $j$ such that $\eta_j\geq 1/2$. This is equivalent to say that 
$$\delta_j\geq p_a(C)-2g_j+1$$
where $\delta_j$ is the number of nodes on $C_j$. 
\end{remark}

To conclude this section  we will focus on two  classes of curves of compact type: chain-like and comb-like 
curves. In these cases we prove the existence of components with small slopes for $0< d \leq s$. 

\begin{proposition}
\label{PROP:Chain}
Let $C$ be a chain-like curve with $\gamma \geq 2$ smooth irreducible components.
Assume that 
$$ s \geq 2(\gamma-1) \qquad \mbox{ and }\qquad \gamma\leq d\leq s.$$
Then, there exists a neighborhood $U$ of the canonical polarization $\un{\eta}$ such that 
for any $\w\in U$, $\Uw(s\cdot \underline{1},d)$ has a non-empty component $X_{d_1,\dots,d_{\gamma}}$ with small slopes.
\end{proposition}

\begin{proof}
We assume that the components of $C$ are ordered "in a natural way" (see Example \ref{EX:chain}), so that $A_j=\bigcup_{i=1}^{j}C_i$ for $j=1,\dots,\gamma-1$. 
We will prove that there exists a non-empty component 
$X_{d_1,\dots, d_{\gamma}} \subset 
\U_{(C,\un{\eta})}(s \cdot {\underline 1},d)$   satisfying the  requests of the Theorem. Then, using Lemma \ref{LEM:STARCONDITIONS}(c) we will obtain the result for a suitable neighborhood of $\underline{\eta}$.
\vspace{2mm}

By Lemma \ref{LEM:STARCONDITIONS}(b), the component $X_{d_1,\dots, d_{\gamma}}$ corresponds to a $\gamma$-uple $(d_1,\dots, d_{\gamma}) \in \mathbb{Z}^{\gamma}$ with $\sum{d_i}=d$ and which satisfies Condition $\eqref{EQ:STARJcanonical}$ for all $j= 1, \dots \gamma -1$.
\vspace{2mm}

For $j=1,\dots, \gamma-1$, consider the system
$$ (\diamondsuit)_j:\qquad \begin{cases}
(\sum_{i=1}^{j-1}d_i)+1 \leq \sum_{i=1}^{j}d_i \leq d - (\gamma -j) \\
d\sum_{i=1}^{j}\eta_i-s/2 < \sum_{i=1}^j d_i < d\sum_{i=1}^{j}\eta_i+s/2-(\gamma-j-1).
\end{cases}
$$
Note that if $(d_1,\dots, d_{j})$ satisfies $(\diamondsuit)_{j}$, then it satisfies  \eqref{EQ:STARJcanonical} too (for the same index $j$). 
\medskip

{\bf Claim}: For all $j=1,\dots, \gamma-1$, there exists $(d_1,\dots,d_j)\in \bN^{j}_{>0}$ which satisfies Condition $(\diamondsuit)_j$. We will prove the claim by recurrence.
\vspace{2mm}

{\bf Step (A)} We prove that there exists an integer $d_1$ satisfying the system 
$$(\diamondsuit)_1:\quad  \begin{cases} 
1 \leq d_1 \leq d-(\gamma -1) \\
d\eta_1-s/2 < d_1 < d\eta_1+s/2-(\gamma-2)
\end{cases}$$
\vspace{2mm}

Note that the system admits real solutions if and only if 
$$(A_1): \quad d\eta_1-s/2 < d-\gamma+1\qquad \mbox{ and }\qquad (A_2):\quad d\eta_1+s/2-\gamma+2> 1.$$
In this case, it is easy to see that we also have integer solutions.
The above conditions follows easily from the assumption $ s \geq 2(\gamma-1)$. Indeed
$$d\eta_1-s/2\leq d\eta_1-(\gamma-1)<d-\gamma+1$$
$$d\eta_1+s/2-\gamma+2\geq d\eta_1+(\gamma-1)-\gamma+2\geq d\eta_1+1>1$$
so this proves Step (A).
\vspace{2mm}

If $\gamma=2$, then we are done. Indeed, since $1\leq d_1\leq d-1$ and $d_1$ satisfies $(\star)_1$: a component satisfying our request is $X_{d_1,d-d_1}$. Hence, from now on, we can assume $\gamma\geq 3$.
\vspace{2mm}

{\bf Step (B)} Assume now that $1\leq j\leq \gamma-2$. We will prove that if $(d_1,\dots,d_i) \in \bN_{>0}^i$ satisfies $(\diamondsuit)_i$ for all $i=1,\dots,j$, then there exists $d_{j+1} \in \bN_{>0}$ such that $(d_1,\dots,d_j,d_{j+1})$ satisfies $(\diamondsuit)_{j+1}$.
\vspace{2mm}

We consider the system  
$$\qquad \begin{cases}
(\sum_{i=1}^{j}d_i)+1 \leq x \leq d - (\gamma -j-1) \\
d\sum_{i=1}^{j+1}\eta_i-s/2 < x < d\sum_{i=1}^{j+1}\eta_i+s/2 - (\gamma-j-2).
\end{cases}
$$
It  admits integer solutions if and only if
$$(B_1): \quad d\sum_{i=1}^{j+1}\eta_i+s/2 -(\gamma-j-2) > 1+\sum_{i=1}^{j}d_i\qquad\mbox{ and } \qquad (B_2):\quad d\sum_{i=1}^{j+1}\eta_i-s/2 < d-(\gamma-j-1).$$

Equation $(B_1)$ follows from $\diamondsuit_{j}$ and the assumption $ s \geq 2(\gamma-1)$:
$$d\sum_{i=1}^{j+1}\eta_i+s/2-(\gamma-j-2)=d\sum_{i=1}^{j}\eta_i+s/2+(\gamma-j-1)+(d\eta_{j+1}+1)>\sum_{i=1}^jd_i+(d\eta_j+1)>\sum_{i=1}^jd_i+1.$$
For $(B_2)$ it is enough to use $ s \geq 2(\gamma-1)$:
$$d\sum_{i=1}^{j+1}\eta_i-s/2<d\sum_{i=1}^{j+1}\eta_i-(\gamma-1)\leq d-(\gamma-1)+j=d-(\gamma-j-1).$$

Let $x$ be an integer solution, 
we set $d_{j+1}= x - \sum_{i=1}^j d_i$.  It follows that 
$1 \leq d_{j+1} \leq d-1 < s$ and that $(d_1,\dots,d_{j+1})$
satisfies $(\star)_{j+1}$. 
\vspace{2mm}

{\bf Step (C)} By recurrence we produce a $\gamma$-uple $(d_1,\dots, d_{\gamma}) \in \bN^{\gamma}_{>0}$ with $\sum_{i=1}^{\gamma}d_i = d$, which satisfies $(\star)_j$ for any $j=1,\dots,\gamma-1$ and with $1 \leq d_i \leq s-1$. Then $X_{d_1,\dots,d_{\gamma}}$ is a non-empty irreducible component of the moduli space $\U_{(C,\eta)}^s(s\cdot \underline{1},d)$.
\vspace{2mm}
\end{proof}

\begin{proposition}
\label{PROP:Comb}
Let $C$ be a comb-like curve with $\gamma \geq 3$ smooth irreducible components. Assume that 
$$s\geq 2(\gamma-1) \qquad \mbox{ and }\qquad \gamma\leq d\leq s.$$
Then, there exists a neighborhood $U$ of the canonical polarization $\un{\eta}$ such that 
for any $\w\in U$, $\Uw(s\cdot \underline{1},d)$ has a non-empty component $X_{d_1,\dots,d_{\gamma}}$ with small slopes.
\end{proposition}

\begin{proof}
We can assume that the components of $C$ are ordered so that $C_{\gamma}$ is the "grip" of $C$, i.e. $C_{\gamma}$ is the component with $\gamma-1$ nodes (see Example \ref{EX:comb}).  As in the previous case, we are looking for a $\gamma$-uple $(d_1,\dots, d_{\gamma}) \in \mathbb{N}_{>0}^{\gamma}$ with $\sum{d_i}=d$ and which satisfies the stability Condition $\eqref{EQ:STARJcanonical}$ for any $j=1,\dots \gamma-1$ for the canonical polarization. 
With the chosen ordering, we have $A_j=C_j$ for all $j=1,\dots,\gamma-1$ so that the above stability condition can be written as
\begin{equation}
\label{EQ:stabcomblike}
d\eta_j-s/2<d_j<d\eta_j+s/2.
\end{equation}
We can assume that $d\eta_j<s/2+1$ for all $j=1,\dots, \gamma$ since, otherwise, we can conclude using Proposition \ref{PROP:TEC}$(b)$. Then, it is easy to see that for all $j=1,\dots,\gamma-1$, $d_j=1$ satisfy the  Inequality \eqref{EQ:stabcomblike}. Since $\gamma-1<d$ by assumption we have that
$X_{1,\dots,1,d-(\gamma-1)}$ is a non-empty irreducible component of the moduli space $\U_{(C,\eta)}^s(s\cdot \underline{1},d)$.
\end{proof}

As a consequence of Propositions \ref{PROP:TEC}, \ref{PROP:Chain} and \ref{PROP:Comb} and Theorem \ref{THM:Main}, we obtain the following result:
\begin{theorem}
\label{THM:MAIN-B}
Let $C$ be a curve of compact type with $\gamma \geq 2$ smooth irreducible components  $C_i$  of genus $g_i \geq 2$.
Let  $d, s$ and $k$ be integers such that $$k \leq  1 + s(g_i -1) \mbox{ for all } i=1,\dots, \gamma.$$
Assume, furthermore, that one of the following conditions holds:
\begin{itemize}
    \item $d,s$ and $\gamma$ satisfy one of the three conditions of Proposition \ref{PROP:TEC}; 
    \item $C$ is either a chain-like or comb-like curve,
    $s\geq 2(\gamma-1)$ and $\gamma\leq d\leq s$.
\end{itemize} 
Then, setting $r=s+k$, the Brill-Noether locus $\Bw(r\cdot \underline{1},d,k)$ is non-empty whenever $\w$ lies in a suitable open neighborhood of the canonical polarization. Moreover, it has an irreducible component of dimension $\beta_C(r,d,k)$.
\end{theorem}

We conclude the section with the following conjecture:

\begin{conjecture}
\label{CONJ:STABJ}
The Brill-Noether locus $\Bw((s+k)\cdot \underline{1},d,k)$ is not empty whenever 
$$2(\gamma-1)\leq s,\quad  \gamma\leq d\leq s \qquad \mbox{ and }\qquad k g_i \leq 1+s(g_i-1), \forall i = 1,\dots \gamma,$$
for any curve of compact type and $\w$ in a suitable neighborhood of $\underline{\eta}$.
\end{conjecture}

\end{document}